\documentclass[a4paper]{amsart}
\usepackage{graphicx}
\usepackage{amssymb,amscd} 
\usepackage{stmaryrd}
\usepackage{dsfont}
\usepackage[mathcal]{euscript}
\usepackage{tikz}
\usepackage{tikz-cd}
\tikzset{labl/.style={anchor=south, rotate=270, inner sep=.5mm}}
\usepackage{hyperref}
\hypersetup{%
  bookmarksnumbered=true,%
  colorlinks=true,%
  linkcolor=blue,%
  citecolor=blue,%
  filecolor=blue,%
  menucolor=blue,%
  urlcolor=blue,%
  bookmarksopen=true,%
  bookmarksdepth=2,%
  pageanchor=true}

\title{Central support for triangulated categories}

\author{Henning Krause}
\address{Fakult\"at f\"ur Mathematik\\
Universit\"at Bielefeld\\ D-33501 Bielefeld\\ Germany}
\email{hkrause@math.uni-bielefeld.de}

\theoremstyle{plain}
\newtheorem{thm}{Theorem}[section]
\newtheorem*{thm1}{Theorem~1}
\newtheorem*{thm2}{Theorem~2}
\newtheorem{prop}[thm]{Proposition}
\newtheorem{lem}[thm]{Lemma} 

\newtheorem{cor}[thm]{Corollary}

\theoremstyle{definition}
\newtheorem{defn}[thm]{Definition}
\newtheorem{exm}[thm]{Example}

\theoremstyle{remark}
\newtheorem{rem}[thm]{Remark}

\numberwithin{equation}{section}

\newcommand{\Ann}{\operatorname{Ann}}

\newcommand{\cha}{\operatorname{char}}

\newcommand{\Coh}{\operatorname{Coh}}
\newcommand{\coh}{\operatorname{coh}}

\newcommand{\colim}{\operatorname*{colim}}

\newcommand{\End}{\operatorname{End}}

\newcommand{\Hom}{\operatorname{Hom}}

\newcommand{\id}{\operatorname{id}}

\newcommand{\Ker}{\operatorname{Ker}}

\renewcommand{\mod}{\operatorname{mod}}
\newcommand{\Mod}{\operatorname{Mod}}

\newcommand{\Perf}{\operatorname{Perf}}

\newcommand{\Pt}{\operatorname{Pt}}

\newcommand{\Spec}{\operatorname{Spec}}
\newcommand{\Spc}{\operatorname{Spc}}

\newcommand{\stmod}{\operatorname{stmod}}

\newcommand{\supp}{\operatorname{supp}}

\newcommand{\thick}{\operatorname{thick}}
\newcommand{\Thick}{\operatorname{Thick}}

\newcommand{\Ab}{\mathrm{Ab}}

\newcommand{\cent}{\mathrm{cent}}

\newcommand{\op}{\mathrm{op}}

\newcommand{\col}{\colon}

\newcommand{\iso}{\xrightarrow{\raisebox{-.4ex}[0ex][0ex]{$\scriptstyle{\sim}$}}}
\newcommand{\kos}[2]{{#1}/\!\!/{#2}} 
\newcommand{\leftiso}{\xleftarrow{\raisebox{-.4ex}[0ex][0ex]{$\scriptstyle{\sim}$}}}

\newcommand{\longiso}{\xrightarrow{\ \raisebox{-.4ex}[0ex][0ex]{$\scriptstyle{\sim}$}\ }}

\newcommand{\lto}{\longrightarrow}
\newcommand{\smatrix}[1]{\left[\begin{smallmatrix}#1\end{smallmatrix}\right]}

\newcommand{\xto}{\xrightarrow}
\newcommand*{\intref}[2]{\def\tmp{#1}\ifx\tmp\empty\hyperref[#2]{\ref*{#2}}\else\hyperref[#2]{#1~\ref*{#2}}\fi}

\def\calO{\mathcal O}
\def\P{\mathcal P}
\def\calS{\mathcal S} 
\def\T{\mathcal T} 
\def\U{\mathcal U}
\def\V{\mathcal V}
\def\W{\mathcal W}
\def\X{\mathcal X}

\def\bfD{\mathbf D}

\def\bbP{\mathbb P}

\def\bbZ{\mathbb Z}

\newcommand{\fra}{\mathfrak{a}} 
\newcommand{\frb}{\mathfrak{b}}

\newcommand{\frp}{\mathfrak{p}}

\def\a{\alpha}
\def\b{\beta}
\def\e{\varepsilon}
\def\d{\delta}
\def\g{\gamma}

\def\Ga{\varGamma}

\def\Si{\Sigma}

\def\one{\mathds 1}

\begin{document}



\begin{abstract}
  For any essentially small triangulated category the centre of its
  lattice of thick subcategories is introduced; it is a spatial frame
  and yields a notion of central support. A relative version of this
  centre recovers the support theory for tensor triangulated
  categories and provides a universal notion of cohomological
  support. Along the way we establish Mayer--Vietoris sequences for
  pairs of commuting subcategories.
\end{abstract}

\date{January 25, 2023}

\maketitle

\section{Introduction}

Triangulated categories are algebraic desiderata as they were
introduced in order to deal with cohomologies arising in geometric or
topological settings. But over the last years triangulated categories
have been turned into geometric objects, thanks to the notion of
support \cite{Ba2005,BIK2008,BKS2007,KP2017,NVY2019}. In this work we pursue
this direction\footnote{We follow the slogan `geometry is algebra is
  geometry', which is an adaptation of Ringel's `algebra is geometry
  is algebra' in \cite{Ri1999}.} and study in which way the geometry
of an essentially small triangulated category $\T$ is captured by its
lattice of thick subcategories, and more precisely by its centre. Our
first result (Corollary~\ref{co:MV}) is meant to illustrate this
because many applications of geometric flavour involve Mayer--Vietoris
sequences \cite{BF2007, BIK2008,Ri1997,Th1997}.

\begin{thm1}
  For  a pair of thick subcategories $\U,\V\subseteq\T$ the following
  are equivalent.
  \begin{enumerate}
\item The pair $\U,\V$ is \emph{commuting}, that is,  any morphism between
objects of $\U$ and $\V$ factors through an object of
$\U\wedge\V$. 
\item For each pair of objects $X,Y\in\T$ there exists a long exact sequence  
\begin{multline*}\cdots\to \Hom_{\T/(\U\wedge\V)}(X,Y)\to  \Hom_{\T/\U}(X,Y)\oplus
 \Hom_{\T/\V}(X,Y)\to\\
 \to\Hom_{\T/(\U\vee\V)}(X,Y)\to\Hom_{\T/(\U\wedge\V)}(X,\Si
 Y)\to\cdots.
\end{multline*}
\end{enumerate}
\end{thm1}

The notion of support for objects of a triangulated category $\T$ amounts
to additional geometric structure. It requires a topological space and
assigns to an object of $\T$ a subset of this space. For instance, the
space could be given by the set of prime tensor ideals provided that
$\T$ admits a symmetric tensor product \cite{Ba2005}, or by the set of
prime ideals of a graded commutative ring acting centrally on $\T$
\cite{BIK2008}.

An essential feature of any notion of support is a parameterisation of
thick subcategories by subsets of a topological space. We
take as input a lattice $T$ of thick subcategories. Then its \emph{centre}
\[Z(T):=\{\U\in T\mid \U \text{ and }\V\text{ commute for all } \V\in T\}\]
provides such a  space because $Z(T)$ is a \emph{spatial frame};
this means it identifies with the lattice of open sets of a
topological space. This space may be used to define a support for any
object such that the central subcategories are parameterised by the
open subsets.

We summarise and state our main result (Theorem~\ref{th:main}) for the
lattice of thick subcategories $\Thick\T$ of an essentially small
triangulated category $\T$.

\begin{thm2}
  Let $T\subseteq \Thick\T$ be a sublattice which is closed under
  arbitrary joins. Then its centre $Z(T)\subseteq T$ is also a
  sublattice and closed under arbitrary joins. Moreover, $Z(T)$ is a
  spatial frame and for any pair $\U,\V$ in $T$ there is a
  Mayer--Vietoris sequence provided at least one of $\U$ and $\V$ is
  central.
\end{thm2}

A consequence is that $T$ is a distributive lattice when $Z(T) = T$,
but the converse is not necessarily true. For example, $Z(T)=T$ for
$T$ the lattice of thick tensor ideals when $\T$ is a rigid tensor
triangulated category (Example~\ref{ex:rigid-tt}). Thus we provide
foundations for \emph{triangular geometry} in the spirit of tensor
triangular geometry \cite{Ba2011,KP2017,NVY2019}, with additional
structure on $\T$ given by a distinguished sublattice
$T\subseteq\Thick\T$.

Our discussion of central thick subcategories is
inspired by recent work of Gratz and Stevenson \cite{GS2022}, and in
particular its presentation at the Abel Symposium 2022 in
\r{A}lesund. A crucial observation in \cite{GS2022} is the relevance
of \emph{distributivity} for the lattice of thick subcategories. In
this work we provide evidence for a close connection between
distributivity and \emph{commutativity}, responding in particular to
the quest for `interesting distributive sublattices of $\Thick\T$' in
\cite[\S8]{GS2022}.

This paper is organised as follows. In \S\ref{se:cohfun} we briefly
discuss our main tool: the category of cohomological functors of a
triangulated category which identifies with its ind-completion.  In
\S\ref{se:centre} we introduce central subcategories and show that
they form a spatial frame. Along the way we establish a couple of
Mayer--Vietoris sequences for any pair of commuting subcategories.
In \S\ref{se:examples} we provide many
examples, including triangulated categories with additional
structure given by a tensor product or a central ring action.  This
motivates a relative version of the lattice of central subcategories
which is discussed in the final \S\ref{se:relative}.

\subsection*{Acknowledgements}
It is a pleasure to thank Greg Stevenson for inspiring discussions
related to this work. Much thanks also to Dave Benson for providing
specific examples from modular representation theory, and to Kent
Vashaw for help with the non-commutative tensor products in this paper.

Part of this work was done during the Trimester Program `Spectral
Methods in Algebra, Geometry, and Topology' at the Hausdorff Institute
in Bonn. It is a pleasure to thank for hospitality and for funding
by the Deutsche Forschungsgemeinschaft (DFG) under Excellence
Strategy EXC-2047/1-390685813.

\section{The category of cohomological functors}\label{se:cohfun}

Triangulated categories were introduced in order to deal with
cohomologies arising in geometric or topological settings. If
we wish to treat triangulated categories as geometric objects in their
own right, it seems natural to study them by cohomological methods,
using the category of cohomological functors.  In this section we
collect some basic and well-known facts.

\subsection*{Cohomological functors}

Let $\T$ be an essentially small triangulated category with
suspension $\Si\colon\T\iso\T$.  Recall that a functor
$\T^\op\to\Ab$ into the category of abelian groups is
\emph{cohomological} if it takes exact triangles to exact sequences.
We denote by $\Coh\T$ the category of cohomological
functors. Morphisms in $\Coh\T$ are natural transformation and the
Yoneda functor $\T\to\Coh\T$ sending $X\in\T$ to
\[H_X=\Hom_\T(-,X)\] is fully faithful.  The suspension $\Si$ extends
to a functor $\Coh\T\iso\Coh\T$ by taking $F$ in $\Coh\T$
to $F\circ\Si^{-1}$; we denote this again by $\Si$.  

It is convenient to view $\Coh\T$ as a full subcategory of the
category $\Mod\T$ of all additive functors $\T^\op\to\Ab$.  Note
that (co)limits in $\Mod\T$ are computed pointwise.  For $E$ and $F$
in $\Mod\T$ we write $\Hom(E,F)$ for the set of morphisms from $E$
to $F$.  Thus $\Hom(H_X,F)\cong F(X)$ for $X$ in $\T$, by Yoneda's
lemma.

Any additive functor $F\colon\T^\op\to\Ab$ can be written
canonically as a colimit of representable functors
\begin{equation*}\label{eq:slice}
\colim_{H_X\to F} H_X \iso F
\end{equation*}
where the colimit is taken over the slice category $\T/F$; see
\cite[Proposition~3.4]{Grothendieck/Verdier:1972a}. Objects in
$\T/F$ are morphisms $ H_X\to F$ where $X$ runs through the objects
of $\T$. A morphism in $\T/F$ from $H_X\xto{\phi} F$ to
$H_{X'}\xto{\phi'} F$ is a morphism $\alpha \colon X\to X'$ in $\T$
such that $\phi' H_{\alpha}=\phi$.

A theorem of Lazard says that a module is flat if and only if it is a
filtered colimit of finitely generated free modules; this has been
generalised to functor categories by Oberst and R\"ohrl.  The
following lemma shows that cohomological and flat functors agree; this
is well-known, for instance from \cite[Lemma~2.1]{Neeman:1992a}.

\begin{lem}
\label{le:flat}
The cohomological functors $\T^\op\to\Ab$ are precisely the filtered
colimits of representable functors (in the category of additive
functors $\T^\op\to\Ab$).  In particular, the category $\Coh\T$
has filtered colimits.\qed
\end{lem}

 \subsection*{Exact functors}

We say that a sequence of morphisms in $\Coh\T$ is \emph{exact}
provided that evaluation at each object in $\T$ yields an exact
sequence in $\Ab$.  This provides an exact structure in the sense of
Quillen on the category $\Coh\T$.

Let $\T$ and $\U$ be
essentially small triangulated categories. A functor
$P\colon\Coh\T\to\Coh\U$ is said to be \emph{exact} if it takes
exact sequences to exact sequences and if there is a natural
isomorphism $P\circ \Si\iso\Si\circ P$.

An exact functor $f\colon\T\to\U$ induces a pair of functors
\[
f^*\colon\Coh\T\lto\Coh\U\qquad\text{and}\qquad
f_*\colon\Coh\U\lto\Coh\T
\] 
where $f^*(F)=\colim_{H_X\to F} H_{f(X)}$ and $f_*(G)=G\circ f$. We
recall some basic facts.

\begin{lem}[{\cite[Lemma~2.4]{BIK2015}}]
\label{le:exact-fun}
\pushQED{\qed}
Let $f\colon\T\to\U$ be an exact functor between essentially small
triangulated categories.
\begin{enumerate} 
\item The functor $f^*$ is a left adjoint of $f_*$.
\item The functors $f^*$ and $f_*$ are exact and preserve filtered
  colimits.
\item If $f$ is fully faithful, then $f^*$ is fully faithful and
  $\id\iso f_*\circ f^*$.
\item If $f$ is a quotient functor, then $f_*$ is fully faithful and
  $f^*\circ f_*\iso\id$.\qedhere
\end{enumerate} 
\end{lem}

Let $\calS\subseteq\T$ be a triangulated subcategory. With $i\col
\calS\to \T$ and $q\col \T\to \T/\calS$ the canonical functors,
set
\begin{equation*}
\label{eq:loc-fun}
\Ga_\calS =i^*\circ i_*\qquad\text{and}\qquad L_\calS=q_*\circ q^*.
\end{equation*} 
These are exact functors on $\Coh \T$. We may identify $\Coh\calS$ and $\Coh\T/\calS$ with full subcategories
of $\Coh\T$. Then we have for each $F\in\Coh\T$
\[F\in\Coh\calS\;\;\iff\;\;\Ga_\calS F\iso F\;\;\iff\;\; L_\calS F=0\]
and
\[F\in\Coh\T/\calS\;\;\iff\;\; F\iso L_\calS F \;\;\iff\;\;\Ga_\calS F=0.\]

The functors $\Ga_\calS$ and $L_\calS$ fit into a long exact
sequence. To explain this we note that for any $F\in \Coh\T$ one has
\begin{equation*}
\label{eq:gam-F}
\Ga_\calS F =  \colim_{H_S\to F} \Hom_{\T}(-,S)
\end{equation*}
where the colimit is taken over the slice category $\calS/F$, which is
filtered because $\calS$ is a triangulated subcategory. On the other
hand,  the
definition of the morphisms in the quotient category $\T/\calS$ gives
for any $X\in\T$
\[
L_\calS H_X=\Hom_{\T/\calS}(-,X)=\colim_{X\to Y}\Hom_\T(-,Y)
\]
where $X\to Y$ runs through all morphisms with cone in $\calS$.
Combining this we obtain the following.

\begin{prop}[{\cite[Proposition~2.10]{BIK2015}}]
\label{pr:loc-seq}
\pushQED{\qed} For a triangulated subcategory $\calS\subseteq\T$ there
is the following long exact sequence of functors $\Coh\T\to\Coh\T$.
\begin{equation}\label{eq:loc-seq}
 \cdots \lto\Si^{n-1} L_\calS\lto \Si^n\Ga_\calS\lto \Si^n\lto \Si^nL_\calS \lto \Si^{n+1}\Ga_\calS\lto\cdots\qedhere
\end{equation}
\end{prop}

The sequence \eqref{eq:loc-seq} is functorial with
respect to $\calS$. Thus each inclusion $\calS_1\subseteq\calS_2$ of
triangulated subcategories of $\T$ induces a morphism between the
corresponding long exact sequences. An immediate consequence is the following.

\begin{lem}[{\cite[Lemma~2.16]{BIK2015}}]
  \label{le:rules-commute}
  \pushQED{\qed}
Let $\calS_1\subseteq \calS_2\subseteq\T$ be triangulated
subcategories and let $(\Ga_1,L_1)$ and $(\Ga_2,L_2)$ be the
corresponding pairs of (co)localisation functors on $\Coh\T$. The
morphisms in \eqref{eq:loc-seq} induce isomorphisms
\[
\Ga_1\Ga_2\cong\Ga_1\cong\Ga_2\Ga_1,\quad L_1L_2\cong L_2\cong
L_2 L_1,\quad \Ga_1 L_2=0=L_2\Ga_1,\quad \Ga_2 L_1\cong L_1\Ga_2.\qedhere
\]
\end{lem}

The rest of this paper makes excessive use of the category of
functors
\[\Coh\T\lto\Coh\T\]
that are exact and perserve filtered colimits. We consider them
\emph{up to isomorphism}, which means that we identify functors $F$
and $G$ when there is a natural isomorphism $F\iso G$. Let us denote
by $\End\T$ the isomorphism classes of objects. Then we have the
following relevant structures:
\begin{enumerate}
\item an associative multiplication $\End\T\times\End\T\to\End\T$ with 
  an identity (given by the composition of functors), 
\item an associative sum $\End\T\times\End\T\to\End\T$ with 
  an identity (given by the direct sum of functors), 
  \item a suspension $\Si\colon\End\T\iso\End\T$, and
\item a notion of exact sequence $F\to G\to H$ with connecting
  morphism $H\to\Si F$, providing a long exact sequence
  \[\cdots\lto \Si^{n-1}H\lto \Si^nF\lto \Si^nG\lto \Si^nH\lto \Si^{n+1} F\lto\cdots.\]
\end{enumerate}

\section{The lattice of central subcategories}\label{se:centre}

Throughout we keep fixed an essentially small triangulated category $\T$.  Let
$\Thick\T$ denote the lattice of thick subcategories of $\T$. For
$\U,\V$ in $\Thick\T$ we write $\U\wedge\V$ for their meet (i.e. the
intersection) and $\U\vee\V$ for their join (i.e\ the smallest thick
subcategory containing $\U$ and $\V$). For a class of objects
$\X\subseteq\T$ let $\thick\X$ denote the smallest thick subcategory
containing $\X$.

\subsection*{Central subcategories}

Given a pair $\U,\V$  of thick subcategories we consider the property
that morphisms between objects from $\U$ and $\V$ factor through objects
in $\U\wedge\V$. This condition appears already in
\cite[II.2.3]{Ve1997}, though the equivalent conditions (2)--(4) in the
lemma below seem to be new.

\begin{lem}\label{le:U-V-central}
  For a pair $\U,\V\in\Thick\T$ the following are equivalent.
\begin{enumerate}
\item Every morphism $\U\ni U\to V\in\V$ factors through an object in
  $\U\wedge\V$. 
\item The canonical morphism $\Ga_{\U\wedge\V}\to\Ga_\U\Ga_\V$ is an
  isomorphism.
\item The canonical morphism $L_{\U\wedge\V}\Ga_\V\to L_\U\Ga_\V$ is
  an isomorphism.
\item The canonical morphism $\Ga_\U L_{\U\wedge\V}\to \Ga_\U L_\V$ is
  an isomorphism.
\item The canonical morphism
  $\Hom_{\T/(\U\wedge\V)}(-,V)\to \Hom_{\T/\U}(-,V)$ is an isomorphism
  for all $V\in\V$.
\end{enumerate}
\end{lem}
\begin{proof}
  When $\Coh\U$ is viewed as a full subcategory of $\Coh\T$, then the
  objects in $\Coh\U$ are precisely the filtered colimits of objects
  $H_U$ with $U\in\U$. From this it follows that $F\in\Coh\T$ belongs
  to $\Coh\U$ if and only if every morphism $H_X\to F$ with $X\in\T$
  factors through $H_U$ for some $U\in\U$.

(1) $\Leftrightarrow$ (2) 
  The  above general observation shows that the first
  two conditions are equivalent to the property that each object of
  the form $\Ga_\U\Ga_\V F$ belongs to $\Coh\U\wedge\V$.

  (2) $\Leftrightarrow$ (3) Composition of $\Ga_\V$ with the exact
  sequences \eqref{eq:loc-seq} given by $\U$ and $\U\wedge\V$ yield
  the following diagram.
\[\begin{tikzcd}[column sep = small]
    \cdots\arrow{r}&
    \Si^{-1}L_{\U\wedge\V}\Ga_{\V}\arrow{r}\arrow{d}{\Si^{-1}\b}&\Ga_{\U\wedge\V}\arrow{r}\arrow{d}{\a}&
    \Ga_{\V}\arrow{r}\arrow{d}{\id}&
    L_{\U\wedge\V}\Ga_{\V}\arrow{r}\arrow{d}{\b}&\Si\Ga_{\U\wedge\V}\arrow{r}\arrow{d}{\Si\a}&\cdots\\
    \cdots\arrow{r}&
    \Si^{-1}L_{\U}\Ga_{\V}\arrow{r}&\Ga_{\U}\Ga_\V\arrow{r}&\Ga_{\V}\arrow{r}&
    L_{\U}\Ga_{\V}\arrow{r}&\Si\Ga_{\U\wedge\V}\arrow{r}&\cdots
\end{tikzcd}\] The five lemma implies that $\a$ is an isomorphism
if and only if $\b$ is an isomorphism.

 (2) $\Leftrightarrow$ (4) Composition of $\Ga_\U$ with the exact
  sequences \eqref{eq:loc-seq} given by $\V$ and $\U\wedge\V$ yield
  the following diagram.
\[\begin{tikzcd}[column sep = small]
    \cdots\arrow{r}&
    \Si^{-1}\Ga_{\U} L_{\U\wedge\V}\arrow{r}\arrow{d}{\Si^{-1}\b'}&\Ga_{\U\wedge\V}\arrow{r}\arrow{d}{\a}&
    \Ga_{\U}\arrow{r}\arrow{d}{\id}&
    \Ga_{\U} L_{\U\wedge\V}\arrow{r}\arrow{d}{\b'}&\Si\Ga_{\U\wedge\V}\arrow{r}\arrow{d}{\Si\a}&\cdots\\
    \cdots\arrow{r}&
    \Si^{-1}\Ga_{\U} L_{\V}\arrow{r}&\Ga_{\U}\Ga_\V\arrow{r}&\Ga_{\U}\arrow{r}&
    \Ga_{\U} L_{\V}\arrow{r}&\Si\Ga_{\U\wedge\V}\arrow{r}&\cdots
\end{tikzcd}\] The five lemma implies that $\a$ is an isomorphism
if and only if $\b'$ is an isomorphism.

(3) $\Leftrightarrow$ (5) For $V\in\V$ the inclusion
$\U\wedge\V\subseteq\U$ induces the morphism
\[\Hom_{\T/(\U\wedge\V)}(-,V)\cong L_{\U\wedge\V}\Ga_\V H_V\lto
  L_\U\Ga_\V H_V\cong \Hom_{\T/\U}(-,V).\]
As any functor of the form $\Ga_\V F$ is a filtered colimit of
representable functors $H_V$, this morphism is an isomorphism for all
$V$ if and only if  $L_{\U\wedge\V}\Ga_\V\iso L_\U\Ga_\V$.
\end{proof}

We note that the equivalent conditions of the preceding lemma imply the
\emph{second Noether isomorphisms}, i.e.\ the canonical functors
\[\V/(\U\wedge\V)\lto(\U\vee\V)/\U\qquad\text{and}\qquad\U/(\U\wedge\V)\lto(\U\vee\V)/\V\]
are equivalences up to direct summands. This is clear for the first
one and follows for the other by passing to the opposite category $\T^\op$. 

\begin{defn}\label{de:centre}
  A pair of elements $\U,\V$ in $\Thick\T$ is called \emph{commmuting}
  if we have canonical isomorphisms
\[\Ga_\U\Ga_\V\leftiso\Ga_{\U\wedge\V}\iso
  \Ga_\V\Ga_\U.\] An equivalent condition is that a morphism between
objects of $\U$ and $\V$ (in either direction) factors
through an object of $\U\wedge\V$.

An element $\U$ in $\Thick\T$ is called \emph{central} if it commutes
with all $\V$ in $\Thick\T$.
\end{defn}

We denote by $Z(\Thick\T)$ the set of central thick subcategories of
$\T$; it is partially ordered by inclusion. In the following we study
the basic properties of $Z(\Thick\T)$. In particular we will see that
it is a sublattice of $\Thick\T$.

We continue with a sequence of technical lemmas which establish
further identities for exact functors $\Coh\T\to\Coh\T$ given by a
pair of thick subcategories. The next one shows that
commutativity is equivalent to an excision property.

\begin{lem}\label{le:commuting}
  For a  pair $\U,\V$ in $\Thick\T$ the following are equivalent.
  \begin{enumerate}
  \item The pair $\U,\V$ is commuting.
\item  The canonical morphism
  \[\psi\colon L_{\U\wedge\V}\Ga_{\U\vee\V}\lto L_{\U}\Ga_{\U\vee\V}\oplus
    L_{\V}\Ga_{\U\vee\V}\] is an isomorphism.
\item The canonical morphism
  \[\psi_X\colon\Hom_{\T/(\U\wedge\V)}(-,X)\lto
    \Hom_{\T/\U}(-,X)\oplus \Hom_{\T/\V}(-,X)\] is an isomorphism for
  all $X\in\U\vee\V$.
\item The canonical morphism
  $L_{\U\wedge\V}\Ga_\V\to L_\U\Ga_{\U\vee\V}$ is an isomorphism.
  \end{enumerate}
\end{lem}
\begin{proof}
  (1) $\Rightarrow$ (2) The objects $X\in\T$ such that
  \[ L_{\U\wedge\V}H_X\lto L_{\U}H_X\oplus L_{\V}H_X\] is an
  isomorphism form a thick subcategory $\X\subseteq\T$. From
  Lemma~\ref{le:U-V-central} it follows that  $\U,\V\subseteq\X$.
  Thus $\psi$ is an isomorphism, as any functor of the form
  $\Ga_{\U\vee\V} F$ is a filtered colimit of representable functors
  $H_X$ with $X\in\U\vee\V$.

  (2) $\Rightarrow$ (3) We have $\psi H_X=\psi_X$ for each $X\in\U\vee\V$.

  (3) $\Rightarrow$ (1) For $X\in\V$ the morphism $\psi_X$ specialises
  to $\Hom_{\T/(\U\wedge\V)}(-,X)\to \Hom_{\T/\U}(-,X)$, and similar
  for $X\in\U$. Thus $\U$ and $\V$ are commuting, again by
  Lemma~\ref{le:U-V-central}.

  (1) $\Rightarrow$ (4)  The objects $X\in\T$ such that
  $L_{\U\wedge\V}\Ga_\V H_X\to L_\U\Ga_{\U\vee\V}H_X$ is an
  isomorphism form a thick subcategory $\X\subseteq\T$. From
  Lemma~\ref{le:U-V-central} it follows that $\U,\V\subseteq\X$.  Thus
  the composite
  \[L_{\U\wedge\V}\Ga_\V\leftiso L_{\U\wedge\V}\Ga_\V\Ga_{\U\vee\V}\to
    L_\U\Ga_{\U\vee\V}\] is an isomorphism, as any functor of the form
  $\Ga_{\U\vee\V} F$ is a filtered colimit of representable functors
  $H_X$ with $X\in\U\vee\V$.

  (4) $\Rightarrow$ (1)  Suppose $L_{\U\wedge\V}\Ga_\V\iso
  L_\U\Ga_{\U\vee\V}$. Composing with $\Ga_\V$ yields   $L_{\U\wedge\V}\Ga_\V\iso
  L_\U\Ga_{\V}$, and therefore $\Ga_{\U\wedge\V}\iso
  \Ga_\U\Ga_\V$ by  Lemma~\ref{le:U-V-central}.  Composing with $\Ga_\U$ yields
  $L_{\U\wedge\V}\Ga_\V\Ga_\U=0$, and therefore
\[\Ga_{\U\wedge\V}\leftiso\Ga_{\U\wedge\V}\Ga_\V\Ga_\U  \iso\Ga_\V\Ga_\U.\]
  Thus the pair $\U,\V$ is commuting.
\end{proof}

\begin{lem}\label{le:U-V-central-localisation}
  Let $\U,\V\in\Thick\T$ be commuting. Then the following holds.
\begin{enumerate}
\item The canonical morphism $L_\U L_\V\to L_{\U\vee\V}$ is an 
      isomorphism. 
\item The canonical morphism $L_\U \Ga_\V\to L_\U\Ga_{\U\vee\V}$ is an 
   isomorphism.
\item The canonical morphism $\Ga_\U L_\V\to \Ga_{\U\vee\V}L_\V$ is an 
   isomorphism.
 \end{enumerate}
\end{lem}
\begin{proof}
 Consider the exact sequences \eqref{eq:loc-seq} given by $\V$
  and $\U\vee\V$. We compose with $L_\U$ and obtain the following
  diagram.
  \[\begin{tikzcd}[column sep = small]
    \cdots\arrow{r}&
    \Si^{-1}L_{\U} L_{\V}\arrow{r}\arrow{d}{\Si^{-1}\d}&L_{\U}\Ga_\V\arrow{r}\arrow{d}{\g}&
    L_{\U}\arrow{r}\arrow{d}{\id}&
    L_{\U}L_{\V}\arrow{r}\arrow{d}{\d}&\Si L_{\U}\Ga_\V\arrow{r}\arrow{d}{\Si\g}&\cdots\\
    \cdots\arrow{r}&
    \Si^{-1}L_{\U\vee\V}\arrow{r}&L_\U\Ga_{\U\vee\V}\arrow{r}&L_{\U}\arrow{r}&
    L_{\U\vee\V}\arrow{r}&\Si L_\U\Ga_{\U\vee\V}\arrow{r}&\cdots
  \end{tikzcd}\] Now consider
the composite
\[L_{\U\wedge\V}\Ga_\V\xto{\b} L_\U\Ga_{\V}\xto{\g}
  L_\U\Ga_{\U\vee\V}.\] Then $\b$ is an isomorphism by
Lemma~\ref{le:U-V-central}, and $\g\b$ is an isomorphisms by
Lemma~\ref{le:commuting}.  Thus $\g$ is an isomorphism, and it follows
from the five lemma that $\d$ is an isomorphism. A similar argument
yields the last isomorphism.
\end{proof}

\begin{lem}\label{le:central-loc-coloc}
  Let $\U,\V\in\Thick\T$ be commuting.  Then we have canonical
  isomorphisms
  \[L_\U L_\V\iso L_{\U\vee\V}\leftiso L_\V L_\U.\]
Also, we have canonical isomorphisms
\[\Ga_\V L_\U\leftiso \Ga_{\V} L_{\U\wedge\V}\cong L_{\U\wedge\V} \Ga_{\V}
  \iso L_\U\Ga_\V\]
and
\[\Ga_\V L_\U\iso \Ga_{\U\vee\V} L_\U\cong L_\U \Ga_{\U\vee\V}
  \leftiso L_\U\Ga_\V.\]
\end{lem}
\begin{proof}
  The first and third sequence of isomorphisms follow from
  Lemma~\ref{le:U-V-central-localisation} and the second sequence from
  Lemma~\ref{le:U-V-central}.
\end{proof}

The following lemma says that meet and join preserve central elements.

\begin{lem}\label{le:central-res}
  Let $\U,\V\in\Thick\T$ and suppose that $\U$ is central.
  \begin{enumerate}
  \item The element $\U\wedge\V$ is central in $\Thick\V$. 
  \item The element $(\U\vee\V)/\V$ is central in $\Thick\T/\V$.
\end{enumerate}
\end{lem}
\begin{proof}
  (1) Let $\X\in\Thick\V$. Then any morphism between objects of $\U\wedge\V$
  and $\X$ (in either direction) factors through an object of
  $\U\wedge\X=(\U\wedge\V)\wedge\X$. 

  (2) Observe that for any $\X\in\Thick\T$ containing $\V$ the functor
  $\Ga_\X$ restricts to a functor $\Coh\T/\V\to\Coh\T/\V$ where it
  identifies with $\Ga_{\X/\V}$, since $\Ga_\X L_\V\cong
  L_\V\Ga_\X$. We obtain using  Lemma~\ref{le:central-loc-coloc}
  \[\Ga_\X\Ga_{\U\vee\V} L_\V\cong \Ga_\X \Ga_{\U} L_\V\cong
    \Ga_\U\Ga_\X L_\V\cong \Ga_\U L_\V \Ga_\X \cong \Ga_{\U\vee\V}
    L_\V \Ga_\X \cong \Ga_{\U\vee\V} \Ga_\X L_\V,\] and therefore
  $\Ga_\X \Ga_{\U\vee\V}\cong \Ga_{\U\vee\V}\Ga_\X$ when restricted to
  $\Coh\T/\V$.
\end{proof}

\subsection*{Mayer--Vietoris sequences}

A pair of commuting subcategories gives rise to a pair of long exact
sequences. It turns out that commutativity is actually equivalent to
having such Mayer--Vietoris sequences.

\begin{prop}\label{pr:MV}
  Let $\U,\V\in\Thick\T$ be commuting.  Then there are canonical exact
  sequences
  \[\cdots\lto\Si^{-1}L_{\U\vee\V}\lto L_{\U\wedge\V}\stackrel{(1)}\lto L_\U\oplus
   L_\V\lto L_{\U\vee\V}\lto\Si
     L_{\U\wedge\V}\lto\cdots \tag{$\lambda_{\U,\V}$}
   \] and
 \[\cdots\lto\Si^{-1}\Ga_{\U\vee\V}\lto\Ga_{\U\wedge\V}\stackrel{(2)}\lto\Ga_\U\oplus\Ga_\V\lto\Ga_{\U\vee\V}\lto\Si
     \Ga_{\U\wedge\V}\lto\cdots \tag{$\g_{\U,\V}$}
   \]
   where $(1)$ and $(2)$ are induced by the inclusions
   $\U\wedge\V\subseteq\U$ and $\U\wedge\V\subseteq\V$.
 \end{prop}

 \begin{proof}
   The exact sequence \eqref{eq:loc-seq} given by $\U$ and composed
   with $L_{\U\wedge\V}\to L_\V$
   yields  a morphisms of  exact sequences.
\[\begin{tikzcd}[column sep = small]
    \cdots\arrow{r}&\Si^{-1}L_\U\arrow{r}\arrow{d}&\Ga_\U
    L_{\U\wedge\V}\arrow{r}\arrow[d, "\sim" labl]&
    L_{\U\wedge\V}\arrow{r}\arrow{d}& L_\U\arrow{r}\arrow{d}&\Si\Ga_\U
    L_{\U\wedge\V}\arrow{r}\arrow[d, "\sim" labl]&\cdots\\
    \cdots\arrow{r} &\Si^{-1}L_{\U\wedge\V}\arrow{r}&\Ga_{\U}L_\V\arrow{r}&L_\V\arrow{r}&
    L_{\U\vee\V}\arrow{r}&\Si\Ga_{\U}L_\V\arrow{r}&\cdots
\end{tikzcd}\] 
Here we use the identification $L_\U L_\V= L_{\U\vee\V}$ from
Lemma~\ref{le:U-V-central-localisation}, and the isomorphisms follow
from Lemma~\ref{le:U-V-central}. Inverting them yields the
connecting morphism $L_{\U\vee\V}\to\Si L_{\U\wedge\V}$, and then a
standard argument turns the diagram into an exact sequence of the form
$\lambda_{\U,\V}$; cf.\ \cite[Lemma~10.7.4]{TD2008}. The proof for the
second sequence is analogous.
\end{proof}

\begin{cor}\label{co:MV}
  For a pair $\U,\V$ in $\Thick\T$ the following are equivalent.
\begin{enumerate}
\item  The pair $\U,\V$ is commuting.
\item There is an exact Mayer--Vietoris sequence $\lambda_{\U,\V}$.
\item For each object $X\in\T$ there is a canonical exact sequence  
\begin{multline*}\cdots\to \Hom_{\T/(\U\wedge\V)}(-,X)\to  \Hom_{\T/\U}(-,X)\oplus
 \Hom_{\T/\V}(-,X)\to\\
 \to\Hom_{\T/(\U\vee\V)}(-,X)\to\Hom_{\T/(\U\wedge\V)}(-,\Si X)\to\cdots.
\end{multline*}
\end{enumerate}
\end{cor}
\begin{proof}
  (1) $\Rightarrow$ (2) This follows from Proposition~\ref{pr:MV}.

  (2) $\Rightarrow$ (3) For $X\in\T$ apply  the Mayer--Vietoris sequence $\lambda_{\U,\V}$  to $H_X$.

  (3) $\Rightarrow$ (1) For $X\in\U\vee\V$ the sequence
  $\lambda_{\U,\V}H_X$  yields an
  isomorphism
  \[\Hom_{\T/(\U\wedge\V)}(-,X)\longiso \Hom_{\T/\U}(-,X)\oplus
    \Hom_{\T/\V}(-,X).\] Thus the pair $\U,\V$ is commuting by
  Lemma~\ref{le:commuting}.
\end{proof}

This result reflects the classical situation given by a pair of
subspaces $A,B\subseteq X=A\cup B$ of a topological space $X$.  In
that case the Mayer--Vietoris sequence for a homology theory is
equivalent to an excision property which is symmetric in $A$ and $B$;
cf.\ \cite[10.7]{TD2008}. The triangulated analogue of this is the
isomorphism $L_{\U\wedge\V}\Ga_\V\iso L_\U\Ga_{\U\vee\V}$ from
Lemma~\ref{le:commuting}. Related (though not equivalent) is the
second Noether isomorphism $\V/(\U\wedge\V)\iso(\U\wedge\V)/\U$.

 There is a plethora of Mayer--Vietoris sequences for triangulated
 categories in the literature \cite{BF2007, BIK2008,Ri1997,Th1997},
 including interesting applications. We recover some of them and note
 that all are based (implicitly) on a pair of commuting thick
 subcategories.  A typical example is given by subcategories
 $\U,\V$ satisfying $\Hom_\T(U,V)=0$ and $\Hom_\T(V,U)=0$ for all $U\in\U$ and
 $V\in\V$.

\subsection*{Distributivity}

Recall that a lattice is \emph{distributive} if for all elements
$x,y,z$ there is an equality
\[(x\wedge z)\vee(y\wedge z)=(x\vee y)\wedge z.\]
An equivalent condition is that for all elements
$x,y,z$ there is an equality
\[(x\vee z)\wedge(y\vee z)=(x\wedge y)\vee z.\]

The lattice $\Thick\T$ is rarely distributive. However, we get
distributivity when we restrict to central elements in this
lattice. The proof requires another technical lemma.

\begin{lem}\label{le:fin-dist}
  Let $\U,\V,\W\in\Thick\T$ and suppose $\W\subseteq\U\wedge\V$.
If  $\Ga_{\W}\iso\Ga_\U\Ga_\V$ then $\W=\U\wedge\V$.
\end{lem}
\begin{proof}
Let $X\in \U\wedge\V$. Then $\Ga_\W
H_X\cong\Ga_\U\Ga_\V H_X\cong H_X$. Thus $H_X\in\Coh\W$ and therefore
$X\in\W$.
\end{proof}

\begin{prop}\label{pr:finite-sup}
  Let $\U,\V\in\Thick\T$ be central. Then
  $\U\vee\V$ and   $\U\wedge\V$ are central. Moreover, we have  for all $\W\in\Thick\T$
  \[(\U\wedge \W)\vee(\V\wedge\W)=(\U\vee\V)\wedge\W\] 
and
  \[(\U\vee \W)\wedge(\V\vee\W)=(\U\wedge\V)\vee\W\]
\end{prop}
\begin{proof}
It is clear that $\U\wedge\V$ is central since
\[\Ga_{\U\wedge\V}\Ga_\W\cong \Ga_{\U}\Ga_\V\Ga_\W\cong \Ga_\W\Ga_{\U}\Ga_\V\cong
  \Ga_\W\Ga_{\U\wedge\V}.\] 
Now we show that $\U\vee\V$ is central and combine this with the proof
  of the first identity.  Consider the following Mayer--Vietoris
  sequences, where the middle one uses that $\U\wedge\W$ and
  $\V\wedge\W$ are commuting by Lemma~\ref{le:central-res}.
  \[\begin{tikzcd}
      \cdots\arrow{r}&\Ga_{\U\wedge\V}\Ga_\W\arrow{r}&\Ga_\U\Ga_\W\oplus\Ga_V\Ga_\W\arrow{r}&
      \Ga_{\U\vee\V}\Ga_\W\arrow{r}&\cdots\\
      \cdots\arrow{r}&\Ga_{\U\wedge\V\wedge\W}\arrow{r}\arrow[u, "\sim" labl]\arrow[d, "\sim" labl]&
      \Ga_{\U\wedge\W}\oplus\Ga_{\V\wedge\W}\arrow{r}\arrow[u, "\sim" labl]\arrow[d, "\sim" labl]&
      \Ga_{(\U\wedge
        \W)\vee(\V\wedge\W)}\arrow{r}\arrow{u}\arrow{d}&\cdots\\
      \cdots\arrow{r}&\Ga_\W \Ga_{\U\wedge\V}\arrow{r}&\Ga_\W \Ga_\U\oplus\Ga_\W\Ga_V\arrow{r}&
      \Ga_\W \Ga_{\U\vee\V}\arrow{r}&\cdots
    \end{tikzcd}\] The vertical isomorphisms are clear since $\U$ and
  $\V$ central. The inclusion
  \[(\U\wedge \W)\vee(\V\wedge\W)\subseteq (\U\vee\V)\wedge\W\] is
  automatic. Thus we can apply Lemma~\ref{le:fin-dist}, and then the
  five lemma yields the assertion for $\U\vee\V$.  We conclude that the central
  elements form a sublattice of $\Thick\T$ which is distributive.
\end{proof}

We extend Proposition~\ref{pr:finite-sup} and establish the same
assertions for arbitrary joins, using  some standard arguments.

\begin{prop}\label{pr:frame}
  Let  $(\U_i)_{i\in I}$ be a family of central elements in $\Thick\T$. Then
  $\bigvee_{i\in I}\U_i$ is central and
  \[\bigvee_{i\in I}(\U_i\wedge \V)=\big(\bigvee_{i\in
      I}\U_i\big)\wedge\V\]
  for all $\V\in\Thick\T$.
\end{prop}
\begin{proof}
  Set $\U=\bigvee_{i\in I}\U_i$. The element
  $\U_J=\bigvee_{i\in J}\U_i$ is central for all finite $J\subseteq I$
  by Proposition~\ref{pr:finite-sup}.  Now choose objects $U\in\U$ and
  $V\in\V$. Then $U\in\U_J$ for some finite $J\subseteq I$ and
  therefore any morphism between $U$ and $V$ factors through an object
  in $\U_J\wedge \V\subseteq \U\wedge \V$. Thus $\U$ is central.

  The inclusion $\bigvee_{i\in I}(\U_i\wedge \V)\subseteq\U\wedge\V$
is clear. For the other inclusion we use again
Proposition~\ref{pr:finite-sup}. In fact any object
$X\in\U\wedge \V$ belongs to $\U_J\wedge V=\bigvee_{i\in J}(\U_i\wedge \V)$ for some
  finite $J\subseteq I$, and therefore $X$ belongs to  $\bigvee_{i\in I}(\U_i\wedge \V)$.
\end{proof}

\subsection*{The frame of central subcategories}

Recall that a lattice is a \emph{frame} if it is complete (i.e.\
arbitrary joins exist) and the following \emph{infinite distributivity}
holds: for all elements $(x_i)_{i\in I}$ and $y$ there is an equality
\[\bigvee_{i\in I}(x_i\wedge y)=\big(\bigvee_{i\in I}x_i\big)\wedge
  y.\] A morphism of frames is a morphism of lattices that preserves
finite meets and arbitrary joins.

To a frame $F$ corresponds its space
$\Pt F$ of points.  A \emph{point} is by definition a frame morphisms
$p\colon F\to\{0,1\}$ and the open sets are of the form
\[U(a):=\{p\in\Pt F\mid p(a)=1\}\qquad (a\in F).\] To a topological
space $X$ we can associate its frame of open subsets $\Omega X$, and
this yields an adjoint pair of functors
\[\begin{tikzcd}[column sep=huge]
\mathbf{Frm} \ar[r,shift left,"\Pt"] & \mathbf{Top}^\op \ar[l,shift
left,"\Omega"].
\end{tikzcd}\] A frame $F$ is \emph{spatial} if the assignment
$a\mapsto U(a)$ gives an isomorphism $F\iso\Omega\Pt F$.  It is
convenient to identify the points in $\Pt F$ with the \emph{prime
  elements} of $F$ via the assignment
\[\Pt F\ni p\longmapsto \bigvee_{p(a)=0} a\in F.\]

The central thick subcategories form a
sublattice \[Z(\Thick\T)\subseteq \Thick\T\] by
Proposition~\ref{pr:finite-sup}.  This lattice is in fact a spatial
frame.

\begin{cor}\label{co:spatial}
The central thick subcategories form a spatial frame, and therefore
the lattice of central thick subcategories identifies with the lattice
of open subsets of a sober topological space.
\end{cor}
\begin{proof}
  The infinite distributivity follows from Proposition~\ref{pr:frame}.
  An application of Zorn's lemma produces sufficiently many prime
  elements such that the frame is spatial; see the proof of
  Theorem~4.7 in \cite{GS2022}.
\end{proof}

The assignment $\T\mapsto Z(\Thick\T)$ is functorial in the following sense.

\begin{cor}
  For any triangulated subcategory $\calS\subseteq\T$ the canonical functors
  $\calS\to\T$ and $\T\to\T/\calS$ induce frame morphisms
  \[Z(\Thick\T)\lto Z(\Thick\calS),\qquad \U\mapsto
    \U\wedge\calS\]
  and
   \[Z(\Thick\T)\lto Z(\Thick\T/\calS),\qquad \U\mapsto (\U\vee\calS)/\calS.\]
\end{cor}
\begin{proof}
  The maps are well defined by Lemma~\ref{le:central-res}. They
  preserve finite meets and arbitrary joins by
  Propositions~\ref{pr:finite-sup} and \ref{pr:frame}.
\end{proof}

\subsection*{Central support}

For an essentially small triangulated category $\T$ let
\[\Spc_\cent\T:=\Pt Z(\Thick\T)\]
denote the space associated with the frame of central thick subcategories of
$\T$.  For any object $X\in\T$ its \emph{central support} is the set
of prime elements $\P$ in $Z(\Thick\T)$ not containing $X$, so
\begin{equation}\label{eq:supp}
  \supp_\cent X:=\{\P\in \Spc_\cent\T\mid X\not\in\P\}.
\end{equation}
For a class of objects $\X\subseteq\T$ we set 
\[\supp_\cent \X:=\bigcup_{X\in\X}\supp_\cent X,\]
which equals the set of primes $\P$ such that $\X\not\subseteq\P$.

\begin{cor}\label{co:supp-bijection}
  \begin{enumerate}
\item  The assignment $\U\mapsto \supp_\cent \U$ induces an isomorphism
  between $Z(\Thick\T)$ and the frame of open subsets of
  $\Spc_\cent\T$.
\item For any pair $\U,\V$ in $\Thick\T$ there is a Mayer--Vietoris
  sequence $\lambda_{\U,\V}$ provided at least one of $\U$ and $\V$
  is central.
\end{enumerate}
\end{cor}
\begin{proof}
  This is a reformulation of Corollaries~\ref{co:MV} and \ref{co:spatial}.
\end{proof}

\section{Examples}\label{se:examples}

Having introduced for any triangulated category the
lattice of central thick subcategories, it is now time to look at examples.

\subsection*{Tensor triangulated categories}
Let $\T=(\T,\otimes,\one)$ be a tensor triangulated category, i.e.\ a
triangulated category equipped with a symmetric monoidal structure
which is exact in each variable. Then $\Coh\T$ inherits a symmetric
monoidal structure which is exact and preserves filtered colimit in
each variable by setting $H_X\otimes H_Y=H_{X\otimes Y}$ for all $X,Y$
in $\T$. The category $\T$ is called \emph{rigid} if there is an exact
functor $D\colon \T^\op\to\T$ and a natural isomorphism
\[\Hom_\T(X\otimes Y,Z)\cong\Hom_\T(Y,DX\otimes Z)\]
for all $X,Y,Z\in\T$. Note that in this case $X\iso D^2X$ for all $X\in\T$.

\begin{lem}\label{le:tensor-coh-ideal}
  Let $\T$ be a rigid tensor triangulated category and $\U\subseteq\T$
  a thick tensor ideal. Then $\Coh\U$ and $\Coh\T/\U$ are tensor
  ideals of $\Coh\T$.
\end{lem}
\begin{proof}
  We begin with $\Coh\U$. Let $E,F\in\Coh\T$ written as filtered
  colimits $E=\colim_i H_{X_i}$ and $F=\colim_j H_{X_j}$. If
  $F\in\Coh\U$, then we may assume $X_j\in\U$ for all $j$. Thus
  $H_{X_i}\otimes F\cong\colim_j H_{X_i\otimes X_j}$ belongs to
  $\Coh\U$ for all $i$ because $\U$ is a tensor ideal, and therefore
  $E\otimes F\cong\colim_i (H_{X_i}\otimes F)$ is in $\Coh\U$.

Now consider $\Coh\T/\U$, where rigidity is used as follows.  Let $F\in \Coh\T/\U$ and
$U,X\in\T$. Then we have
  \[(H_X\otimes F)(U)\cong\Hom(H_U,H_X\otimes F)\cong\Hom(H_{DX\otimes
    U},F)\cong F(DX\otimes U)=0\]
when $U\in\U$. For an arbitrary $E=\colim_i H_{X_i}$ in $\Coh\T$ we
conclude that $E\otimes F$ lies in $\Coh\T/\U$, since
\[(E\otimes F)(U)\cong\colim_i (H_{X_i}\otimes F)(U)=0\qquad\text{for}\qquad U\in\U.\qedhere\]
\end{proof}

\begin{prop}\label{pr:tensor-central}
  Let $\T$ be a rigid tensor triangulated category.  Then any pair of
  thick tensor ideals is commuting.
\end{prop}
\begin{proof}
  We provide two proofs; the second one is due to Kent Vashaw
  \cite{Va2023} and does not use the commutativity of the tensor
  product.
  
  Let $\U,\V\in\Thick\T$ be tensor ideals. From
  Lemma~\ref{le:tensor-coh-ideal} we know that $\Coh\U$ and
  $\Coh\T/\U$ are tensor ideals of $\Coh\T$.  From this we deduce
  isomorphisms
  \[\Ga_\U\cong (\Ga_\U H_\one)\otimes -\qquad\text{and}\qquad
    L_\U\cong (L_\U H_\one)\otimes - \] as follows.  Consider for
  $F\in\Coh\T$ the exact sequence
  \[\cdots\lto (\Ga_\U H_\one)\otimes F\lto F\lto (L_\U H_\one)\otimes F\lto\cdots\]
  and apply  $\Ga_\U$. We have
\[\Ga_\U((\Ga_\U H_\one)\otimes F)\cong (\Ga_\U H_\one)\otimes F\qquad\text{and}\qquad
  \Ga_\U((L_\U H_\one)\otimes F)=0\] because
\[(\Ga_\U H_\one)\otimes F\in\Coh\U \qquad\text{and}\qquad(L_\U
  H_\one)\otimes F\in\Coh\T/\U\] by the first part of this proof. Thus
$(\Ga_\U H_\one)\otimes F\iso \Ga_\U F$, and the proof of the
isomorphism $(L_\U H_\one)\otimes F\iso L_\U F$ is analogous. Now the
commutativity $\Ga_\U\Ga_\V\cong \Ga_\V\Ga_\U$ follows.

The second proof is more direct. Given $U\in\U$ and $V\in\V$, any
morphism $U\to V$ admits a factorisation
\[U\to U\otimes DU\otimes U\to V\otimes DU\otimes U\to V\] with
$V\otimes DU\otimes U$ in $\U\wedge\V$.  By symmetry the same holds
for a morphism $V\to U$. Thus the pair $\U,\V$ is commuting.
\end{proof}

\begin{rem}
  (1) In Example~\ref{ex:A2} one finds a pair of thick tensor ideals
  which is not commuting.

  (2) The second proof shows that the proposition remains true when
  the tensor product is not commutative \cite{Va2023}.

  (3) For a predecessor of the proposition, see Proposition~6.5 in \cite{BKSS2020}.
\end{rem}
  
The following immediate consequence explains why in several
interesting examples all thick subcategories are central. This
includes the category of perfect complexes over a commutative ring or
the stable homotopy category of finite spectra.

\begin{cor}
  If a tensor triangulated category is generated by its tensor
  identity then all thick subcategories are central.
\end{cor}
\begin{proof}
  The assumption means $\T=\thick\one$ and this implies that all thick
  subcategories are tensor ideals. In particular, $\Coh\U$ and
  $\Coh\T/\U$ are tensor ideals of $\Coh\T$ for all thick
  $\U\subseteq\T$.  Now apply the argument from
  Proposition~\ref{pr:tensor-central}.
\end{proof}

\begin{rem}
   If one assumes rigidity the corollary  remains true when the tensor product
  is not commutative, using the second proof
  of  Proposition~\ref{pr:tensor-central}.
\end{rem}

\subsection*{Central actions}

Let $R$ be a graded commutative ring. We consider only graded
$R$-modules and $\Spec R$ denotes the set of homogeneous prime ideals
of $R$ with the Zariski topology.  We fix a triangulated category $\T$ with an $R$-linear action
via a homomorphism of graded rings $R\to Z^*(\T)$ into the graded
centre of $\T$; cf.\ \cite{BIK2008,BF2008}. This means that for all
objects $X,Y$ in $\T$
\[\Hom_\T^*(X,Y):=\bigoplus_{n\in\bbZ}\Hom_\T(X,\Si^n Y)\]
is a graded $R$-module. Also, for $F\in\Coh\T$ and $X\in\T$ the graded abelian group
\[F^*(X):=\bigoplus_{n\in\bbZ}F(\Si^{-n} X)\] is an $R$-module.

Let $\Phi$ be a multiplicatively closed set of homogeneous elements in
$R$. For an $R$-module $M$ set
\[M [\Phi^{-1}]:=M\otimes_R R [\Phi^{-1}]\] and note that
$M[\Phi^{-1}]=0$ if and only if $M_\frp=0$ for all $\frp\in\Spec R$
such that $\frp\cap\Phi=\varnothing$. We define a triangulated
category $\T[\Phi^{-1}]$ as follows. The objects are the same as in
$\T$. For the morphisms set
\[\Hom^*_{\T[\Phi^{-1}]}(X,Y):=\Hom^*_{\T}(X,Y)[\Phi^{-1}].\]
The exact triangles in $\T[\Phi^{-1}]$ are
given by the ones isomorphic to images of exact triangles in $\T$
under the canonical functor $\T\to \T[\Phi^{-1}]$. The kernel of this
functor is the thick subcategory
\[\T_\Phi:=\{X\in\T\mid H^*_X[\Phi^{-1}]=0\}\]
and from \cite[Lemma~3.5]{BIK2015} we have a triangle equivalence
\begin{equation}\label{eq:mult-closed}
  \T/\T_\Phi\longiso \T[\Phi^{-1}].
\end{equation}

A special case of the above construction arises for each $\frp\in\Spec
R$. We denote by $\T_\frp$ the triangulated category given by
localising the morphisms at $\frp$,
so \[\Hom^*_{\T_\frp}(X,Y):=\Hom^*_{\T}(X,Y)_\frp.\]
The notation $X\mapsto X_\frp$ is used for the canonical functor $\T\to\T_\frp$.

\begin{lem}\label{le:mult-closed-central}
 A thick subcategory of the form $\T_\Phi$  is central.
\end{lem}
\begin{proof}
  To simplify notation we set $\Ga_\Phi:=\Ga_{\T_\Phi}$ and
  $L_\Phi:=L_{\T_\Phi}$. Note that
  \[(L_\Phi F)^*\cong F^*[\Phi^{-1}]\] for each $F$ in $\Coh\T$ by
  \eqref{eq:mult-closed}.  Now fix a thick subcategory
  $\U\subseteq\T$. We have canonical isomorphisms
\[L_\Phi \Ga_\U\leftiso \Ga_\U L_\Phi \Ga_\U\iso\Ga_\U L_\Phi\]
since $L_\Phi$ restricts to a functor $\Coh\U\to\Coh\U$. It remains to
consider the following commutative diagram with exact rows.
  \[\begin{tikzcd}
      \cdots\arrow{r}&\Ga_\Phi\Ga_\U\arrow{r}&\Ga_\U\arrow{r}&
      L_\Phi\Ga_\U\arrow{r}&\cdots\\
      \cdots\arrow{r}&\Ga_\U\Ga_\Phi\Ga_\U\arrow{r}\arrow{u}\arrow{d}&
      \Ga_\U\Ga_\U\arrow{r}\arrow[u, "\sim" labl]\arrow[d, "\sim" labl]&
      \Ga_\U L_\Phi\Ga_\U\arrow{r}\arrow[u, "\sim" labl]\arrow[d, "\sim" labl]&\cdots\\
      \cdots\arrow{r}&\Ga_\U \Ga_\Phi\arrow{r}&\Ga_\U\arrow{r}&
      \Ga_\U L_\Phi\arrow{r}&\cdots\\
    \end{tikzcd}\] 
Then the five lemma yields the isomorphism $\Ga_\Phi\Ga_\U\cong
\Ga_\U\Ga_\Phi$. Thus $\T_\Phi$ is central.
\end{proof}

For an ideal $\fra$ of $R$ that is generated by finitely many
homogeneous elements set
\[V(\fra):=\{\frp\in\Spec R\mid \fra\subseteq\frp\}\]
and consider the thick subcategory
\[\T_{V(\fra)}:=\{X\in\T\mid X_\frp=0 \text{ for } \frp\in (\Spec R)\setminus V(\fra)\}.\]

\begin{lem}\label{le:Zariski-closed-central}
  A thick subcategory of the form $\T_{V(\fra)}$ is
  central.
\end{lem}
\begin{proof}
  Choose homogeneous generators $r_1,\ldots,r_n$ of $\fra$
  and set $\Phi_i=\{r_i^p\mid p\geq 1\}$ for $1\le i\le n$. Then
  $\T_{V(r_i)}=\T_{\Phi_i}$ is central for each $i$ by
  Lemma~\ref{le:mult-closed-central}.  We have
  $V(\fra)=V(r_1)\cap\cdots\cap V(r_n)$ and therefore
  \[\T_{V(\fra)}=\T_{V(r_1)}\wedge\cdots\wedge \T_{V(r_n)}\] is
  central by Proposition~\ref{pr:finite-sup}.
\end{proof}

The sets $V(\fra)$ given by finitely generated ideals form a distributive lattice with
\[V(\fra)\cap V(\frb)=V(\fra+\frb)\qquad\text{and}\qquad V(\fra)\cup V(\frb)=V(\fra\frb).\]

\begin{lem}\label{le:union}
  For $U=V(\fra)$ and $V=V(\frb)$ we have
  \[\T_{U\cap V}=\T_U\wedge\T_V\qquad\text{and}\qquad
  \T_{U\cup V}=\T_U\vee\T_V.\]
\end{lem}
\begin{proof}
  The first equality is clear. For the second equality choose
  homogeneous generators
  $r_1,\ldots,r_m$ of $\fra$ and  $s_1,\ldots,s_n$ of $\frb$ such that
  \[V(\fra)=V(r_1)\cap\cdots\cap V(r_m)\qquad\text{and}\qquad\text
    V(\frb)=V(s_1)\cap\cdots\cap V(s_n).\]
  For a homogeneous element $r\in R$ of degree $|r|$ and an object
  $X\in\T$ we denote by
  $\kos{X}{r}$ the cone of $X\xrightarrow{r}\Si^{|r|} X$ and have
  \[\T_{V(r)}=\thick(\{\kos{X}{r}\mid X\in\T\})\]
  by \cite[Proposition~3.10]{BIK2015}. For  homogeneous elements $r,s\in R$ an application of the
  octahedral axiom shows that $\kos{X}{rs}$ is an extension of $\kos{X}{r}$ and  $\kos{\Si^{|r|}X}{s}$. Thus
  \[\T_{V(r)\cup V(s)}=\T_{V(rs)}=\T_{V(r)}\vee\T_{V(s)}.\]
  Using the distributivity  in $Z(\Thick\T)$
  from Proposition~\ref{pr:finite-sup} we obtain
\[    \T_{V(\fra)}\vee\T_{V(\frb)}=\bigwedge_{i,j}( \T_{V(r_i)}\vee  \T_{V(s_j)})
=\bigwedge_{i,j} \T_{V(r_i s_j)}=\T_{V(\fra\frb)}=\T_{V(\fra)\cup
  V(\frb)}\]
since
\[V(\fra\frb)=\bigcap_{i,j} V(r_i s_j).\qedhere\]
\end{proof}

Recall that a set $V\subseteq \Spec R$ has a Zariski open and
quasi-compact complement if and only if $V=V(\fra)$
for an ideal $\fra$ of $R$ that is generated by finitely many
homogeneous elements.  A subset $V\subseteq\Spec R$ is called
\emph{Thomason} if it can be written as $V=\bigcup_i V_i$ such that
each $(\Spec R)\setminus V_i$ is Zariski open and quasi-compact. The
Thomason subsets are precisely the open subsets for the \emph{Hochster
  dual topology} on $\Spec R$ \cite{Ho1969}. We write $\Spec^{\vee} R$
to denote the spectrum of prime ideals with this dual topology.

Given a frame $F$, one calls $x\in F$ \emph{finite} or \emph{compact}
if $x\le\bigvee_{i\in I}y_i$ implies $x\le\bigvee_{i\in J}y_i$ for
some finite subset $J\subseteq I$. A frame is \emph{coherent} when
every element can be written as a join of finite elements and the
finite elements form a sublattice. In this case we write $F^c$ for the
sublattice of finite elements and observe that for any frame $F'$ a
lattice morphism $F^c\to F'$ extends uniquely to a frame morphism
$F\to F'$.

The following proposition identifies the thick subcategories of $\T$ which are
given by a central ring action; they are central elements in $\Thick\T$.

\begin{prop}\label{pr:Thomason}
  The assignment
  \[V\longmapsto \T_{V}:=\bigvee_{V(\fra)\subseteq V} \T_{V(\fra)}\]
induces a frame morphism from the
  open subsets of $\Spec^{\vee} R$ to the central thick subcategories
  of $\T$.
\end{prop}
\begin{proof}
  The open sets of $\Spec^{\vee} R$ form a coherent frame and the
  finite elements are precisely the sets of the form $V(\fra)$. Then
  the assertion follows from the preceding discussion, once we observe
  that a thick subcategory of the form $\T_{V(\fra)}$ is central by
  Lemma~\ref{le:Zariski-closed-central} and that the map
  $V(\fra)\mapsto\T_{V(\fra)}$ is a lattice morphism by
  Lemma~\ref{le:union}.
\end{proof}

It would be interesting to know when for a given triangulated category
all central subcategories are given by a central ring action. In
algebraic examples this seems to be more likely, while it
seems not to be true for the stable homotopy category of finite
spectra.

\subsection*{Noetherian cohomology}

Let $R$ be a graded commutative ring and $\T$ a triangulated category
with an $R$-linear action via a homomorphism of graded rings $R\to
Z^*(\T)$. We assume that $R$ is noetherian and that $\Hom_\T^*(X,Y)$
is a noetherian $R$-module for all objects $X,Y$ in $\T$.

The \emph{cohomological support} of an object $X$ is by definition
\[\supp_R X:=\{\frp\in\Spec R\mid X_\frp\neq 0\}.\]

\begin{lem}
  We have $\supp_R X=V(\fra)$ where $\fra$ denotes the kernel of the
  ring homomorphism $R\to\End^*_\T(X)$.
\end{lem}
\begin{proof}
  Observe that $X_\frp=0$ if and only $(H^*_X)_\frp=0$. Then the equality
  follows from the fact that for each $Y\in\T$ the $R$-action on
  $H^*_X(Y)=\Hom_\T^*(Y,X)$ factors through $\End_\T^*(X)$, and that
  for any finitely generated $R$-module $M$
\[\supp_R M :=\{\frp\in\Spec R\mid M_\frp\neq 0\}=V(\Ann_R M).\qedhere\]
\end{proof}

\begin{lem}\label{le:noetherian}
  For  an object  $X\in\T$ and a Thomason set $V\subseteq\Spec R$ we
  have
  \[\supp_R X\subseteq V\;\;\iff\;\; X\in\T_V.\]
\end{lem}
\begin{proof}
From the preceding lemma we have an ideal $\fra$ such that   $\supp_R X=V(\fra)$. Then $\supp_R X\subseteq V$ implies
$X\in\T_{V(\fra)}\subseteq\T_\V$. On the other hand, $X\in\T_V$
implies $X\in\T_{V(\frb)}$ for some $V(\frb)\subseteq V$, and therefore
$\supp_R X\subseteq V(\frb)\subseteq V$. 
\end{proof}

\subsection*{Hereditary algebras}

For a finite dimensional hereditary algebra we consider the category of
perfect complexes. Thick subcategories that are generated by
exceptional sequences have  been classified in this case via
non-crossing partitions \cite{HK2016,IT2009}. The following lemma shows that
they are usually not central.

\begin{lem}\label{le:thick-adjoints}
Let $\U\subseteq\T$ be a thick subcategory such that the inclusion
admits a left and a right adjoint. If $\U$ is central, then $^\perp\U=\U^\perp$.
\end{lem}

\begin{proof}
  We have
  \[\Ker\Ga_{^\perp\U}=\Coh\U\qquad\text{and}\qquad\Ker\Ga_\U=\Coh{\U^\perp}.\]
  Thus $\Ga_{^\perp\U}\Ga_\U=0$, while  $\Ga_\U\Ga_{^\perp\U}\neq 0$
  when $^\perp\U\not\subseteq\U^\perp$. On the other hand, we have
  $\Ker L_\U=\Coh\U$. Thus $L_\U L_{^\perp\U}=0$, and  $L_\U
  L_{^\perp\U}\neq 0$ when  $^\perp\U\not\supseteq\U^\perp$.
\end{proof}

\begin{exm}\label{ex:A2}
  We consider the path algebra $A=k(\circ\to\circ)$ over a field $k$
  and the bounded derived category $\T=\bfD^b(\mod A)$ of finitely
  generated $A$-modules, which identifies with the category of perfect
  complexes over $A$. The Hasse diagram of the lattice of thick
  subcategories is the following.
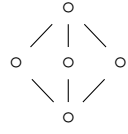
\[\begin{tikzcd}[column sep = tiny, row sep = small]
    &\circ\arrow[dash]{d}\arrow[dash]{ld}\arrow[dash]{rd}\\
    \circ \arrow[dash]{rd}&\circ\arrow[dash]{d}&\circ \arrow[dash]{ld}\\
    &\circ
  \end{tikzcd}\]
For any proper thick subcategory $\U\subseteq\T$ we have
\[\Thick\T=\{0,\U,{^\perp\U},\U^\perp,\T\}.\]
In particular, $Z(\Thick\T)=\{0,\T\}$ by Lemma~\ref{le:thick-adjoints}.

The category of $A$-modules admits a symmetric monoidal structure
given by the pointwise tensor product over $k$. All thick
subcategories are tensor ideals, except the one generated by the
tensor unit.  This provides a pair of thick tensor ideals $\U,\V$
such that $\Ga_\U\Ga_\V\not\cong \Ga_\V\Ga_\U$, because we have
non-zero morphisms between $\U$ and $\V$ in one direction but $\U\wedge\V=0$.
\end{exm}

The above example can be generalised as follows.

\begin{exm}
Let $R$ be a commutative ring and consider the category of perfect
complexes $\T=\Perf T_2(R)$ over the matrix ring
$T_2(R)=\smatrix{R&R\\0&R}$. This category admits an obvious
$R$-action which provides an embedding
\[\Thick\Perf R\hookrightarrow Z(\Thick\T)\subsetneq\Thick\T.\]
\end{exm}  

\begin{exm}
  We consider the category of coherent sheaves on the projective line
  $\bbP^1$ over a fixed field $k$ and its derived category
  $\T=\bfD^b(\coh\bbP^1)$; it is equivalent to the category of perfect
  complexes over the path algebra
  $k(\circ\,
  \genfrac{}{}{0pt}{}{\raisebox{-1.75pt}{$\to$}}{\raisebox{1.75pt}{$\to$}}\,
  \circ)$.

  We recall from \cite{KS2019} the description of $\Thick\T$. There is a lattice isomorphism
\begin{displaymath} 
\{\text{thick subcategories of }\bfD^b(\coh
\mathbb{P}^1) \text{ generated by vector bundles}\}\longiso  \mathbb{Z}
\end{displaymath}
given by $\thick\calO(n)\mapsto n$, where $\mathbb{Z}$ denotes the lattice given by the following Hasse
diagram:
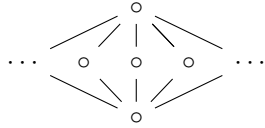
\[\begin{tikzcd}[column sep = tiny, row sep = small]
    &&\circ\arrow[dash]{d}\arrow[dash]{lld}\arrow[dash]{ld}\arrow[dash]{rd}\arrow[dash]{rd}\arrow[dash]{rrd}\\
  \cdots \arrow[dash]{rrd} &\circ \arrow[dash]{rd}&\circ\arrow[dash]{d}&\circ \arrow[dash]{ld}&\cdots \arrow[dash]{lld} \\
    &&\circ
  \end{tikzcd}\]
We have
\[(\thick\calO(n))^\perp=\thick\calO(n-1)\quad\text{and}\quad ^\perp(\thick\calO(n))=\thick\calO(n+1).\]
Thus a subcategory of the form $\thick\calO(n)$ is not central by
Lemma~\ref{le:thick-adjoints}.
On the other hand, we  have a lattice isomorphism
\[
\{\text{thick tensor ideals of }\bfD^b(\coh \mathbb{P}^1)\}\longiso
\{\text{Thomason subsets of }\mathbb{P}^1\}
\]
given by $\U\mapsto\bigcup_{X\in\U}\supp X$. This yields an
isomorphism\footnote{Let $L',L''$ be a pair of lattices with smallest
  elements $0',0''$ and greatest elements $1',1''$. Then
  $L'\amalg L''$ denotes the new lattice which is obtained from the
  disjoint union $L'\cup L''$ (viewed as sum of posets) by identifying
  $0'=0''$ and $1'=1''$.}
\[\Thick \bfD^b(\coh\bbP^1)\longiso \{\text{Thomason subsets of 
  }\mathbb{P}^1\}\amalg\bbZ.\]
Now let $\U\subseteq\T$ be a proper thick tensor ideal and
$\V=\thick\calO(n)$. Then $\U\wedge\V=0$ but $\Ga_\V\Ga_\U\neq 0$,
because we have $\Hom(\calO(n),X)\neq 0$ for any non-zero torsion sheaf $X$
in $\U$. Thus $\U$ is not central. We conclude that  $Z(\Thick\T)=\{0,\T\}$. 
\end{exm}

\subsection*{Group algebras}

Let $G$ be a finite group and $k$ a field. We consider the stable
category $\T=\stmod kG$ of finitely generated modules over the group
algebra $kG$. This is a rigid tensor triangulated category and the tensor
ideals are precisely the thick subcategories given by the
central action of the group cohomology \cite{BCR1997}.

\begin{exm}[Benson]
  In general there are central thick subcategories which are not
  tensor ideal. Let $B_0=e(kG)$ denote the principal block of
  $kG$ which is given by a central idempotent $e\in kG$. We write
  $R:=Z(kG)$ for the centre and set  $V:=V((1- e)R)\subseteq\Spec R$. Then $\stmod B_0=\T_V$ is
  central, and this is a tensor ideal if and only if $B_0=kG$. For
  example, we have $\thick k=\stmod B_0\neq\T$ for $G=A_7$ and
  $\cha k=2$.
\end{exm}

\begin{exm}[Benson]
  Let $G=\bbZ/3\times S_3$ and $\cha k=3$. Then $\thick k$ is not
  central. There are two simple modules, denoted by $k$ and $\e$ (for
  sign representation). There is a uniserial module $X=[k,\e,k]$
  supported in the nucleus \cite{Be1994} and a non-zero morphism $k\to
  X$. However, the intersection of the thick subcategories $\thick k$ and
  $\thick X$ is zero.
\end{exm}

It seems interesting to find out more about $Z(\stmod kG)$ and its
associated space. When $Z(\stmod kG)$ equals the lattice of thick tensor
ideals, then the corresponding space is given by the projective
variety of the group cohomology ring $H^*(G,k)$, but the general case
requires a bigger space and is not clear. For instance, are all
central subcategories given by the central action of a graded
commutative ring?

\section{The relative perspective}\label{se:relative}

Let $\T$ be an essentially small triangulated category. The examples
suggest to get back to our Definition~\ref{de:centre}, introducing
the following relative version of the lattice of central thick
subcategories.

\begin{defn}\label{de:centre-relative}
  Let $T\subseteq\Thick\T$ be a sublattice which is closed under
  finite meets and arbitrary joins, containing in particular $0$ and
  $\T$. Then $\U\in T$ is called \emph{central relative to $T$} if
  $\U$ commutes with all $\V$ in $T$. An equivalent condition is that
  a morphism between objects of $\U$ and any $\V\in T$ (in either
  direction) factors through an object of $\U\wedge\V$.  We set
\[Z(T):=\{\U\in T\mid \U\text{ is central relative to } T\}.\]
\end{defn}

\begin{thm}\label{th:main}
  Let $T\subseteq\Thick\T$ be a sublattice which is closed under
  arbitrary joins. Then its centre $Z(T)\subseteq T$ is also a
  sublattice and closed under arbitrary joins. Moreover, $Z(T)$ is a
  spatial frame and for any pair $\U,\V$ in $T$ the Mayer--Vietoris
  sequence $\lambda_{\U,\V}$ exists provided at least one of $\U$ and $\V$
  is central.
\end{thm}
\begin{proof}
  We have already seen the proof for $T=\Thick\T$. The arguments given
  in Propositions~\ref{pr:finite-sup} and \ref{pr:frame} as well as in
  Corollary~\ref{co:spatial} carry over without change. For the
  Mayer--Vietoris sequence, see Corollary~\ref{co:MV}.
\end{proof}

A consequence is that $T$ is a distributive lattice when  $Z(T)=T$. The converse
is not necessarily true; cf.\ Example~\ref{ex:A2}.

\begin{exm}\label{ex:rigid-tt}
  Let $\T$ be a rigid tensor triangulated category (not necessarily
  commutative) and $T\subseteq\Thick\T$ the sublattice of all thick
  tensor ideals. Then $Z(T)=T$ by
  Proposition~\ref{pr:tensor-central}. In particular, the space
  associated with this frame
    \[\Spc_\otimes\T:=\Pt T\] coincides (up to Hochster duality) with
   Balmer's space $\Spc\T$ \cite{Ba2005,KP2017} and its
   non-commutative analogue \cite{GS2022,NVY2019}.
\end{exm}

\begin{exm}
  Consider the graded centre $Z^*(\T)$ of $\T$ and set
  \[T:=\{\T_V\mid V\subseteq\Spec Z^*(\T) \text{ is Thomason}\}.\] Then
  $T\subseteq Z(\Thick\T)$ is a frame by Proposition~\ref{pr:Thomason}
  and in particular $Z(T)=T$.  One may call the associated space
  \[\Spc_{\coh}\T:=\Pt T\]
  the \emph{cohomological spectrum} of $\T$; it is universal in the
  sense that any central action of a graded commutative ring $R$ via a
  homomorphism $R\to Z^*(\T)$ induces a continuous map
  $\Spc_{\coh}\T\to\Spec^{\vee} R$. In particular, the open subsets of
  $\Spc_{\coh}\T$ parameterise the thick subcategories of $\T$ that
  are given by the central action of any graded commutative ring.
\end{exm}

For a rigid tensor triangulated category $\T$ there is an interesting
relation between $\Spc_\otimes\T$ and $\Spc_{\coh}\T$; for a discussion of
this we refer to \cite{Ba2010}.

For any choice of $T\subseteq\Thick\T$ there is an obvious notion of
\emph{central support relative to $T$} for objects in $\T$, which is
the analogue of the central support defined via \eqref{eq:supp} and
parameterises the elements of $Z(T)$; cf.\
Corollary~\ref{co:supp-bijection}. In particular, one may formulate a
universal property of central support which takes $T$ as input. The
details are left to the interested reader.


\begin{thebibliography}{10}

\bibitem{Ba2005} P. Balmer, The spectrum of prime ideals in tensor
  triangulated categories, J. Reine Angew. Math. {\bf 588} (2005),
  149--168.

\bibitem{Ba2010} P. Balmer, Spectra, spectra, spectra---tensor
  triangular spectra versus Zariski spectra of endomorphism rings,
  Algebr. Geom. Topol. {\bf 10} (2010), no.~3, 1521--1563.

\bibitem{Ba2011} P. Balmer, Tensor triangular geometry, in {\it
    Proceedings of the International Congress of
    Mathematicians. Volume II}, 85--112, Hindustan Book Agency, New
  Delhi, 2010.
  
\bibitem{BF2007} P. Balmer\ and\ G. Favi, Gluing techniques in
  triangular geometry, Q. J. Math. {\bf 58} (2007), no.~4, 415--441.

\bibitem{Be1994} D. J. Benson, Cohomology of modules in the principal
  block of a finite group, New York J. Math. {\bf 1} (1994/95),
  196--205.

\bibitem{BCR1997} D. J. Benson, J. F. Carlson\ and\ J. Rickard, Thick
  subcategories of the stable module category, Fund. Math. {\bf 153}
  (1997), no.~1, 59--80.
  
\bibitem{BIK2008}  D. Benson, S. B. Iyengar\ and\ H. Krause, Local
  cohomology and support for triangulated categories,
  Ann. Sci. \'{E}c. Norm. Sup\'{e}r. (4) {\bf 41} (2008), no.~4,
  573--619.
  
\bibitem{BIK2015} D. Benson, S. B. Iyengar\ and\ H. Krause, A
  local-global principle for small triangulated categories,
  Math. Proc. Cambridge Philos. Soc. {\bf 158} (2015), no.~3,
  451--476.

\bibitem{BKS2007} A.B. Buan, H. Krause, and \O{}. Solberg, Support
  varieties–an ideal approach, Homology, Homotopy Appl., 9 (2007),
  45–74.

\bibitem{BKSS2020} A.~B. Buan,  H. Krause, N. Snashall, and \O{}. Solberg, Support varieties---an
  axiomatic approach, Math. Z. {\bf 295} (2020), no.~1-2, 395--426.
  
\bibitem{BF2008} R.-O. Buchweitz\ and\ H. Flenner, Global Hochschild
  (co-)homology of singular spaces, Adv. Math. {\bf 217} (2008),
  no.~1, 205--242.
  
\bibitem{GS2022} S. Gratz and G. Stevenson, Approximating triangulated
  categories by spaces, arXiv:2205.13356.

\bibitem{Grothendieck/Verdier:1972a} A. Grothendieck and
  J. L. Verdier, Pr\'efaisceaux, in {\it SGA 4, Th\'eorie des Topos et
    Cohomologie Etale des Sch\'emas, Tome 1. Th\'eorie des Topos},
  Lect. Notes in Math., vol. 269, Springer, Heidelberg, 1972-1973,
  pp. 1-184.

\bibitem{Ho1969} M. Hochster, Prime ideal structure in commutative
  rings, Trans. Amer. Math. Soc. {\bf 142} (1969), 43--60.
  
\bibitem{HK2016} A. Hubery\ and\ H. Krause, A categorification of
  non-crossing partitions, J. Eur. Math. Soc. (JEMS) {\bf 18} (2016),
  no.~10, 2273--2313.
  
\bibitem{IT2009} C. Ingalls\ and\ H. Thomas, Noncrossing partitions
  and representations of quivers, Compos. Math. {\bf 145} (2009),
  no.~6, 1533--1562.

\bibitem{KP2017} J. Kock\ and\ W. Pitsch, Hochster duality in derived
  categories and point-free reconstruction of schemes,
  Trans. Amer. Math. Soc. {\bf 369} (2017), no.~1, 223--261.

\bibitem{KS2019} H. Krause\ and\ G. Stevenson, The derived category of
  the projective line, in {\it Spectral structures and topological
    methods in mathematics}, 275--297, EMS Ser. Congr. Rep, EMS
  Publ. House, Z\"{u}rich, 2019.

\bibitem{NVY2019} D. K. Nakano, K. B. Vashaw\ and\ M. T. Yakimov,
  Noncommutative tensor triangular geometry, Amer. J. Math. {\bf 144}
  (2022), no.~6, 1681--1724.
  
\bibitem{Neeman:1992a} A. Neeman, The Brown representability theorem
  and phantomless triangulated categories, J. Algebra {\bf 151}
  (1992), no.~1, 118--155.

\bibitem{Ri1997} J. Rickard, Idempotent modules in the stable
  category, J. London Math. Soc. (2) {\bf 56} (1997), no.~1, 149--170.

\bibitem{Ri1999} C. M. Ringel, Algebra at the turn of the century,
  Lecture at the National Conference on Algebra VII, Beijing Normal
  University, October 1999, Southeast Asian Bull. Math. {\bf 25}
  (2001), no.~1, 147--160.
  
\bibitem{Th1997} R. W. Thomason, The classification of triangulated
  subcategories, Compositio Math. {\bf 105} (1997), no.~1, 1--27.

\bibitem{TD2008} T. tom Dieck, {\it Algebraic topology}, EMS Textbooks
  in Mathematics, European Mathematical Society (EMS), Z\"{u}rich,
  2008.

\bibitem{Va2023}  K. B. Vashaw, Private communication.

\bibitem{Ve1997} J.-L. Verdier, {\it Des cat\'egories d\'eriv\'ees des
    cat\'egories ab\'eliennes}, Ast\'erisque, 239, Soci\'et\'e
  Math\'ematique de France, 1996.

\end{thebibliography}
\end{document}